\documentclass[11pt,reqno]{amsart}

\usepackage{amsmath, amsfonts, amsthm, amssymb, graphicx}
\usepackage{float,epsf,subfigure}
\usepackage{hyperref}

\usepackage{scalerel,stackengine}
\stackMath
\newcommand\reallywidehat[1]{%
\savestack{\tmpbox}{\stretchto{%
  \scaleto{%
    \scalerel*[\widthof{\ensuremath{#1}}]{\kern-.6pt\bigwedge\kern-.6pt}%
    {\rule[-\textheight/2]{1ex}{\textheight}}
  }{\textheight}%
}{0.5ex}}%
\stackon[1pt]{#1}{\tmpbox}%
}
\theoremstyle{plain}
\newtheorem{theorem}{Theorem}[section]

\newtheorem{proposition}{Proposition}[section]

\theoremstyle{definition}
\newtheorem{definition}{Definition}[section]

\theoremstyle{remark}

\numberwithin{equation}{section}

\renewcommand{\div}{\text{div}}

\newcommand{\dd}{\rm d}

\renewcommand{\u}{{\bf u}}
\newcommand{\m}{{\bf m}}
\newcommand{\x}{{\bf x}}

\newcommand{\M}{{\bf M}}

\newcommand{\R}{{\mathbb R}}

\newcommand{\f}{{\bf f}}
\renewcommand{\v}{{\bf v}}

\newcommand{\ee}{{\bf e}}
\newcommand{\kk}{{\bf k}}
\newcommand{\w}{{\bf w}}
\newcommand{\y}{{\bf y}}
\newcommand{\varphib}{\boldsymbol{\varphi}}

\newcommand{\TP}{{\mathbb T_P}}

\def\be{\begin{equation}}
\def\ee{\end{equation}}
\def\bes{\begin{equation*}}
\def\ees{\end{equation*}}
\def\bc{\begin{cases}}
\def\ec{\end{cases}}

\begin{document}
\title[Kolmogorov-Type Theory, Compressible Turbulence and Inviscid Limit]
{Kolmogorov-Type Theory of Compressible Turbulence and Inviscid Limit of
the Navier-Stokes Equations in $\R^3$}

\author{Gui-Qiang G. Chen \and James Glimm}
\address{Gui-Qiang G. Chen, Mathematical Institute, University of Oxford,
         Oxford, OX2 6GG, UK}
\email{chengq@maths.ox.ac.uk}
\address{James Glimm, Department of Applied Mathematics and Statistics,
         Stony Brook University, Stony Brook, NY 11794-3600, USA}
\email{glimm@ams.sunysb.edu}

\subjclass[2010]{35Q30, 76N10, 76F02, 76N17, 65M12, 35L65}
\date{\today}

\keywords{Inviscid limit, convergence, compressible fluids, Navier-Stokes equations, vanishing
viscosity, compressible Kolmogorov-type hypotheses, turbulence, equicontinuity, existence,
weak solutions, Euler equations, numerical convergence.}
\maketitle

\begin{abstract}
We are concerned with the inviscid limit of the Navier-Stokes equations
to the Euler equations for compressible
fluids in $\R^3$.
Motivated by the Kolmogorov hypothesis (1941) for incompressible
flow, we introduce a
Kolmogorov-type hypothesis for barotropic flows, in which the density and
the sonic speed normally vary significantly.
We then observe that the compressible
Kolmogorov-type hypothesis implies
the uniform boundedness of some
fractional derivatives of the weighted velocity and sonic speed
in the space variables in $L^2$,
which is independent of the viscosity coefficient $\mu>0$.
It is shown that this key observation yields
the
equicontinuity in both space and time
of the density in $L^\gamma$ and the momentum in $L^2$,
as well as the uniform bound of the density in $L^{q_1}$ and the velocity in $L^{q_2}$
independent of $\mu>0$, for some fixed $q_1>\gamma$ and $q_2>2$,
where $\gamma>1$ is the adiabatic exponent.
These results lead to the strong convergence of solutions of the Navier-Stokes equations
to a solution of the Euler equations for barotropic fluids in $\R^3$.
Not only do we offer a framework for mathematical existence theories,
but also we offer a framework for the interpretation of numerical
solutions through the identification
of a function space in which convergence should
take place, with the bounds that are independent of $\mu>0$, that is in the high
Reynolds number limit.
\end{abstract}
\maketitle

\section{Introduction}

\medskip
The purpose of this paper is to establish a framework for the existence
theory for the Euler equations
for compressible fluids.
Motivated by a physically well accepted
hypothesis by Kolmogorov \cite{Kolmogorov1,Kolmogorov2}, \textit{Assumption} (K41),
for incompressible flow,
we introduce a corresponding
compressible Kolmogorov-type hypothesis (CKH)
for barotropic flows, where the density and the sonic speed normally vary significantly,
which play an essential
role in compressible turbulence.
The K41 scaling laws
are derived  from a dimensional analysis which pertains to the Euler equations; see \cite{Bat55}.

A further contribution of this paper is that a weaker version \textit{Assumption} (CKHw) of
\textit{Assumption} (CKH) is sufficient to establish our main result
that the convergence of weak solutions of the Navier-Stokes equations
(in the sense of Definition \ref{dfn:ns} as introduced below)
through a subsequence to a solution of the Euler equations
is guaranteed.
In particular, \textit{Assumption} (CKHw), combined with the uniform energy bound,
for the weak solutions
of the Navier-Stokes
equations
with finite-energy initial data
yields some uniform high integrability of the weak solutions with respect
to the viscosity coefficients ({\it i.e.},  the high Reynolds number).

As is well-known, there are two types of turbulence: driven turbulence by a forcing function
and transient turbulence by (strong) initial conditions. The same flow
(in a turbulent wind tunnel or turbulent flow in a pipe) could be
of either type depending on how the system is modeled:
If the pipe is considered in isolation, the flow might be forced;
if the force is a flow connection from a reservoir to the wind tunnel
or pipe, and the reservoir is part of the model,
then the turbulence arises from initial conditions, and the turbulence
will die out when the reservoir is exhausted and no longer drives the
flow. Thus, the distinction between the two (driven turbulence and transient)
is to some extent a matter of points of view and of modeling convenience.
In this paper, our framework is focused on transient turbulence, though the
forced turbulence is also included.
For some recent developments in the mathematical study of energy dissipation
in body-forced turbulence for incompressible flow, see Constantin-Doering \cite{CD},
Doering-Foias \cite{DF}, and the references cited therein.

Our main result is to address the issue of  function spaces and norms in which the convergence might occur.
The rate of dissipation of kinetic energy, $\epsilon$, in the sense of fully developed turbulence,
{\it i.e.}, in the present context, for the Euler equations, is understood as the cascade
of transfer of energy to progressively smaller scales, according to the K41 theory.
In a Navier-Stokes context,
this dissipation cascade terminates with the dissipation
of kinetic energy into heat, {\it i.e.}, molecular motions.
The question remains as to in what norm or function space to express this dissipation cascade
property and convergence results, and in which the convergence of statistical ensemble
occurs.
Our main result addresses the question of function spaces and norms, which is developed from
our previous work on the incompressible flow in \cite{CG12};
while the distinction among time,
space, space-time, and ensemble averages relates to the well-known ergodic hypothesis
and is out of the scope of the present paper.

We refer to \cite{Sun17} for a theoretical analysis of scaling laws for
fully developed compressible fluid turbulence and to \cite{WanYanShi13}
for highly resolved numerical simulations of compressible turbulence. We
do not attempt to survey the large literature on this topic, but additional
references can be traced from these two. The central complication in
passing from incompressible to compressible turbulence is the presence
of additional dimensionless variables, so that scaling relations,
rather than yielding  a closed form expression for the decay of the
velocity spectrum, as in \textit{Assumption} (K41), now possess an additional dimensionless
expression to be resolved by experiment or simulation. This additional
variable enters in the prefactor, $\epsilon^{\frac{2}{3}}$ in (K41), and does not
affect the exponent $-\frac{5}{3}$. Additionally, there are a variety of choices
for the spectrum to be analyzed. Common choices are $\u$, $\sqrt{\rho}\u$,
and
$\sqrt[3]{\rho}\u$. The scaling exponent in these studies is never larger than
$-\frac{5}{3}$, and is usually smaller. The choice $\sqrt[3]{\rho}\u$ is motivated by
studies of high order structure functions.

Mathematically, our result has the
status of an informed conjecture and mathematically rigorous
consequences of this conjecture. Numerical analysts may find
the framework useful, in view of the many difficulties involved in
assessing the convergence of numerical simulations of
turbulent and turbulent mixing flows.
We expect that many physicists will
accept the conclusions as being correct,
even if unproven mathematically. There has
been some discussion regarding the Kolmogorov exponent $\frac{5}{3}$ which occurs in
\textit{Assumption} (CKH).
We note that the main results (if not the detailed estimates) are
not sensitive to this specific number, and corrections (as conventionally
understood) to it due to intermittency do not affect our result.
In fact, our rigorous argument works for an even weaker version, \textit{Assumption} (CKHw).
Most physicists would accept the stronger hypothesis, which we label (CKHi) ($i$ denotes intermittency),
in which the assumed exponent is larger than
$\frac{5}{3}$.
This distinction has an important consequence as discussed in Section $6$ that, for (CKHi) for incompressible flow,
the energy is constant in time.

Not only do we offer a framework for mathematical existence theories,
but also we offer a framework for the interpretation of numerical solutions
of the Navier-Stokes equations.
Only for very modest problems and with the largest
computers can converged solutions of the Navier-Stokes equations
be achieved. These
solutions are called direct numerical solutions (DNS). For most solutions
of interest to science or engineering, the large eddy simulations
(LES) or Reynolds Averaged Navier-Stokes (RANS) simulations are required.
We discuss here the more accurate LES methodology. Briefly, it
employs a numerical grid which will resolve some but not all of the
turbulent eddies. The smallest of those, below the level of the grid
spacing, are not resolved. However, either (a) the Navier-Stokes equations
are modified with additional subgrid scale (SGS) terms to model the
influence of the unresolved scales on those that are being computed
or (b) the numerical algorithm is modified in some manner to
accomplish this effect in some other way. The present article contributes
to this analysis through the introduction of a function space in
which convergence should take place, with bounds that are independent
of the viscosity coefficients, that is in the high Reynolds number limit.

Because of the common occurrence of high Reynolds numbers in flows of
practical and scientific interest and the need to perform LES simulations
to achieve scientific understanding and engineering designs, we observe
that the existence theories for the Euler equations are relevant to the
mathematical theories of numerical analysis.

We further remark that, for incompressible flow,
Onsager's conjecture \cite{Onsager} states
that weak solutions of the Euler equations for incompressible fluids
in $\R^3$ conserve energy only if they have a certain minimal smoothness of the
order of $\frac{1}{3}$--fractional derivatives in $\x\in\R^3$ and that they dissipate
energy if they are rougher (also see \cite{CCFS}).
For barotropic flows, the appearance of shock waves makes the energy
always dissipative, so weak solutions of the Euler equations for compressible fluids
in $\R^3$ do not conserve energy in general.

More precisely,
consider the following Navier-Stokes equations for compressible fluids in $\R^3$:
\begin{equation}\label{eq:ns}
\left\{\begin{aligned}
&\partial_t\rho+\nabla\cdot (\rho\u)=0,\\
&\partial_t(\rho\u)+\nabla\cdot(\rho \u\otimes \u)
+\nabla p=\nabla\cdot \Sigma +\rho \f,
\end{aligned}\right.
\end{equation}
with Cauchy data:
\begin{equation}\label{1.2}
(\rho, \u)|_{t=0}=(\rho_0(\x), \u_0(\x)),
\end{equation}
where $\rho$ is the fluid density, $\u$ the velocity,  $p$ the pressure,
$\Sigma=\Sigma(\nabla\u)$  the deviatoric stress tensor,
$\nabla$  the gradient with respect to the space variable $\x\in\R^3$,
$\u\otimes\u=(u_iu_j)$  the $3\times 3$ matrix for $\u=(u_1, u_2, u_3)$,
and $\f=\f(t,\x)$  a given external force.
The general form of the stress tenor $\Sigma=\Sigma(\nabla\u)$ is
\begin{equation}\label{1.1}
\Sigma(\nabla\u):=2\mu D(\nabla\u)+\lambda (\nabla\cdot\u) {\bf I}_{3\times 3} \qquad\mbox{with}\,\, D(\nabla\u):=\frac{1}{2}\big(\nabla\u+(\nabla\u)^\top\big),
\end{equation}
where $\max\{\mu, |\lambda|\}\in (0, \mu_0)$ for some $\mu_0>0$, $\lambda+\frac{2}{3}\mu\ge 0$, $\lambda=\lambda(\mu)\to 0$ as $\mu\to 0$,
and ${\bf I}_{3\times 3}$ is the $3\times 3$ identity
matrix. Mathematically, it suffices to require that $\lambda+2\mu>0$ with $\lambda=\lambda(\mu)\to 0$ as $\mu\to 0$.

For ideal barotropic fluids, the pressure-density relation is
\begin{equation}\label{1.2d}
p=p(\rho):=\kappa \rho^\gamma,
\end{equation}
where $\gamma>1$ is the adiabatic exponent and $\kappa>0$ is a constant,
and the internal energy is
\begin{equation}\label{1.3}
e=\frac{p}{(\gamma-1)\rho}=\frac{c^2}{\gamma-1},
\end{equation}
where $c$ is the sonic speed.
Then the total energy is
\begin{equation}\label{1.4}
E=\frac{1}{2}\rho|\u|^2+\rho e
=\rho\big(\frac{1}{2}|\u|^2+\frac{c^2}{\gamma-1}\big)
=\frac{1}{2} |\sqrt{\rho}\u|^2
 +\frac{1}{\gamma-1}\big(\sqrt{\rho}c\big)^2.
\end{equation}

For clarity of presentation, we focus on periodic solutions with
period $\TP=[-\frac{P}{2}, \frac{P}{2}]^3\subset\R^3, P>0$;
that is,
$$
(\rho^\mu, \u^\mu)(t,\x+P{\bf e}_i)=(\rho^\mu, \u^\mu)(t,\x)
$$
with $\{{\bf e}_i\}_{i=1}^3$ the canonic basis in $\R^3$. Other cases
can be analyzed correspondingly.

Throughout this paper, we always assume that $\f\in L^1_{\rm loc}(\R_+; L^{\frac{2\gamma}{\gamma-1}}(\TP))$ is periodic
in $\x$ with period $\TP$.
Assume that $(\rho_0,\u_0)$ are  periodic in $\x\in\R^3$ with period $\TP$ so that
\begin{eqnarray}
&& \rho_0\ne 0, \qquad  \rho_0\ge 0 \,\,\, \mbox{{\it a.e.} in $\Omega$},\label{1.5a}\\
&& \m_0=\rho_0\u_0\in L^{\frac{2\gamma}{\gamma+1}}, \label{1.5aa}
\end{eqnarray}
and that the initial total energy is finite in $\TP$:
\begin{eqnarray}\label{1.5b}
\mathcal{E}_0=\int_{\TP}\Big(\frac{1}{2}\rho_0|\u_0|^2+\frac{p(\rho_0)}{\gamma-1}\Big)\, \dd\x<\infty.
\end{eqnarray}

In this paper, we focus on the following class of periodic weak
solutions $(\rho^\mu, \u^\mu)=(\rho^\mu,\u^\mu)(t,\x)$ with period $\TP$ of
the Cauchy problem \eqref{eq:ns}--\eqref{1.2}.

\begin{definition}[Weak solutions of the Navier-Stokes equations]\label{dfn:ns}
A periodic vector function $(\rho^\mu, \u^\mu)=(\rho^\mu,\u^\mu)(t,\x)$ with period $\TP$ is
a weak solution of the Cauchy problem \eqref{eq:ns}--\eqref{1.2}, provided that
$(\rho^\mu, \u^\mu)$ satisfies the following{\rm :}

\begin{enumerate}
\item[\rm (i)]
For any $T>0$,  $(\rho^\mu,\u^\mu)$ satisfies
the equations in \eqref{eq:ns}
in the sense of distributions in $\R_T^3=[0, T)\times \R^3$, the initial condition \eqref{1.2},
and the following properties:
\begin{eqnarray}
&& \u^\mu\in L^2(0,T; H^1(\TP)), \qquad \nabla\u^\mu\in L^2([0, T)\times \TP),\nonumber\\
&& \rho^\mu\in C([0,T]; L_{\rm w}^\gamma(\TP))\cap L^p([0,T)\times \TP) \qquad \mbox{for $1\le p\le \max\{\frac{5\gamma-3}{3},1\}$},\nonumber\\
&&\m^\mu:=\rho^\mu \u^\mu\in C([0,T]; L_{\rm w}^{\frac{2\gamma}{\gamma+1}}(\TP)), \qquad \rho^\mu|\u^\mu|^2\in L^\infty(0,T;  L^1(\TP)), \label{1.6a}
\end{eqnarray}
where a vector $\v\in C([0,T]; L_{\rm w}^q(\TP))$ means that $\v\in L^\infty(0, T; L^q(\TP))$ and
$\v$ is continuous in $t$ with values in $L^q(\TP)$ endowed with the weak topology;

\smallskip
\item[\rm (ii)] The following inequality holds in the sense of distributions in $\R_T^3$:
\begin{align}
&\partial_t E^\mu+\nabla\cdot\big((E^\mu+p^\mu)\u^\mu\big)
+\mu|\nabla\u^\mu|^2+(\lambda+\mu) |\nabla\cdot \u^\mu|^2
\le\nabla\cdot\big(\Sigma^\mu(\nabla\u^\mu)\u^\mu\big)+\m^\mu\cdot \f.
\nonumber\\
&\label{energy-ineq}
\end{align}
\end{enumerate}
\end{definition}

The global existence of weak solutions of the Cauchy problem \eqref{eq:ns}--\eqref{1.2} in the sense of
Definition {\rm \ref{dfn:ns}}
was first established by P.-L. Lions  \cite{PLions} for $\gamma>\frac{9}{5}$,
and was extended by Feireisl-Novotny-Petzeltova \cite{FNP} for $\gamma\in (\frac{3}{2}, \frac{9}{5}]$.
Some further {\it a priori} estimates and properties of solutions of the Navier-Stokes equations \eqref{eq:ns}
for compressible fluids can be found in \cite{Feireisl,PLions} and the references cited therein.
Thus, throughout the paper, we assume either $\gamma>\frac{3}{2}$ or the existence of weak solutions satisfying (i)--(ii) and
$\gamma > 1$.

The energy inequality \eqref{energy-ineq} implies that
\begin{align}
&\int_{\TP}\Big(\frac{1}{2}|\sqrt{\rho^\mu}\u^\mu|^2
 +\frac{1}{\gamma-1}\big(\sqrt{\rho^\mu} c^\mu\big)^2\Big)(t,\x) \,\dd\x
+\int_0^t\int_{\TP} \big(\mu|\nabla\u^\mu|^2+(\lambda+\mu)|\nabla\cdot\u^\mu|^2\big)\,\dd\x \dd t \nonumber\\
&\le \mathcal{E}_0 +\int_0^t\int_{\TP} \m^\mu\cdot\f \,\dd\x \dd t <\infty,\label{energy-ineq-2}
\end{align}
where $\mathcal{E}_0$ is the initial energy over period $\TP$, independent of the viscosity coefficient $\mu$:
$$
\mathcal{E}_0:=\int_{\TP}\Big(\frac{1}{2}|\sqrt{\rho_0}\u_0|^2
 +\frac{1}{\gamma-1}\big(\sqrt{\rho_0}c_0\big)^2\Big)(t,\x) \,\dd\x.
$$
Notice that there exists $C_\gamma>0$ depending only on $\gamma$ such that
$$
|\m^\mu\cdot \f| \le \frac{1}{2}|\sqrt{\rho^\mu}\u^\mu|^2
 +\frac{1}{\gamma-1}\big(\sqrt{\rho^\mu}c^\mu\big)^2 + C_\gamma |\f|^{\frac{2\gamma}{\gamma-1}}.
$$
Then, for any $T>0$, we use the Gronwall inequality to conclude that there exists a constant $M_T$ that is independent of the viscosity coefficient $\mu$,
but may depend on $T$, $\gamma>1$, $\mathcal{E}_0$, and $\|\f\|_{L^1(0,T; L^{\frac{2\gamma}{\gamma-1}}(\TP))}$, such that
\begin{align}
&\int_{\TP}\Big(\frac{1}{2}|\sqrt{\rho^\mu}\u^\mu|^2
 +\frac{1}{\gamma-1}\big(\sqrt{\rho^\mu}c^\mu\big)^2\Big)(t,\x) \,\dd\x
+\int_0^t\int_{\TP} \big(\mu|\nabla\u^\mu|^2+(\lambda+\mu)|\nabla\cdot\u^\mu|^2\big)\,\dd\x \dd t \nonumber\\[2mm]
&\le M_T<\infty.
\label{energy-ineq-3}
\end{align}
From now on, we always denote $M_T$ as a universal constant, independent of the viscosity coefficient $\mu>0$,
but may depend on $T$, $\gamma>1$, $\mathcal{E}_0$, and
$\|\f\|_{L^1(0, T; L^{\frac{2\gamma}{\gamma-1}}(\TP))}$.

In this paper, we consider all the weak solutions $(\rho^\mu, \u^\mu)$ of the Navier-Stokes equations \eqref{eq:ns}
with initial condition \eqref{1.2} in the sense of
Definition {\rm \ref{dfn:ns}} {\it for all} $\gamma>1$
and study their strong convergence to a weak solution $(\rho, \u)$
of the following Euler equations for barotropic compressible fluids in $\R^3$:
\begin{equation}\label{eq:euler}
\left\{\begin{aligned}
&\partial_t\rho+\nabla\cdot (\rho\u)=0,\\
&\partial_t(\rho\u)+\nabla\cdot(\rho \u\otimes \u)
+\nabla p=\rho \f,
\end{aligned}\right.
\end{equation}
with the same Cauchy data \eqref{1.2} in the following sense:

\begin{definition}[Weak solutions of the Euler equations]\label{dfn:euler}
A periodic vector function $(\rho, \m)=(\rho,\rho\u)(t,\x)$ with period $\TP$ is
a weak solution of the Cauchy problem  \eqref{1.2} and \eqref{eq:euler}, provided that
$(\rho, \m)$ satisfies the following{\rm :}

\begin{enumerate}
\item[\rm (i)] For any $\varphi\in C_0^\infty(\R_+\times \R^3;\R)$,
\begin{equation}\label{euler-1}
\int_0^\infty\int_{\R^3}\big(\rho \varphi_t
+ \m \cdot \nabla \varphi\big)\, \dd\x \dd t
+\int_{\R^3} \rho_0(\x)\varphi(0,\x)\, \dd\x
=0;
\end{equation}

\item[\rm (ii)]
For any $\varphib\in C_0^\infty(\R_+\times \R^3;\R^3)$,
\begin{eqnarray}
&&\int_0^\infty\int_{\R^3}\big(\m\cdot \varphib_t
+ \frac{\m\otimes\m}{\rho}:\nabla \varphib
+ p(\rho)\,\div \varphi
+\rho \f\cdot \varphib\big)\, \dd\x \dd t\nonumber\\
&& +\int_{\R^3} \m_0(\x)\cdot \varphib(0,\x)\, \dd\x
=0,\label{euler-2}
\end{eqnarray}
where $A:B$ is the matrix product $\sum_{i,j} a_{ij}b_{ij}$ for $A=(a_{ij})$
and $B=(b_{ij})${\rm ;}

\item[\rm (iii)]  For  all $t\in (0, \infty)$,
\begin{equation}\label{euler-3}
\int_{\TP} E(t,\x)\, \dd\x
\le \int_{\TP}  E_0(\x)\,\dd\x
+\int_0^t\int_{\TP} \m(s,\x)\cdot\f(s,\x)\, \dd\x \dd s.
\end{equation}
\end{enumerate}
\end{definition}

\smallskip
By contrast, less is known regarding the existence theory for the Euler
equations \eqref{eq:euler} with the Cauchy data \eqref{1.2}
even in the weak sense of Definition \ref{dfn:euler}.
For the compressible case,
the analysis of \cite{GW,Neustupa}
gives weakly convergent subsequences, but
the limit
is not shown to satisfy the original equations, in that the interchange
of limits with nonlinear terms in the equations is not justified in this
analysis. On this basis, we state that the
existence of solutions of the Cauchy problem \eqref{1.2} for the Euler equations \eqref{eq:euler}
in $\R^3$ is open as is the
convergence of the inviscid limit from the Navier-Stokes to the Euler equations.
On the other hand,
the Euler equations are fundamental for turbulence.

\medskip
\section{The Kolmogorov-Type Hypothesis for Compressible Turbulence}

\medskip
From the energy estimate \eqref{energy-ineq-3} of weak solutions
$(\rho^\mu, \u^\mu)$ of
the Navier-Stokes equations \eqref{eq:ns} with finite-energy initial data \eqref{1.2} in the sense
of Definition \ref{dfn:ns},
for any $T>0$,
there exists $M_T>0$ independent of $\mu$ such that
\begin{eqnarray}
\|(\sqrt{\rho^\mu}\u^\mu, \sqrt{\rho^\mu}c(\rho^\mu))\|_{L^\infty (0,T; L^2(\TP))}^2
+\sqrt{\mu} \|\nabla\u^\mu\|_{L^2([0,T)\times\TP)}
\le M_T<\infty.
\label{energy-estimate}
\end{eqnarray}
From now on, $M_T$ is always denoted as a universal constant, independent of $\mu>0$,
but may depend on $T$, $\gamma>1$, $\mathcal{E}_0$,
and $\|f\|_{L^1(0,T; L^{\frac{2\gamma}{\gamma-1}}(\TP)}$.

Then the total energy $\mathcal{E}(t)$ per unit mass at time $t$
for isotropic turbulence is:
\begin{align}
\mathcal{E}(t)=&\,\frac{1}{|\TP|}\int_{\TP}
\Big(\frac{1}{2}|\sqrt{\rho^\mu}\u^\mu|^2+\frac{1}{\gamma-1}\big(\sqrt{\rho^\mu}c(\rho^\mu)\big)^2\Big)\dd\x\nonumber\\[2mm]
=&\, \sum_{k\ge 0} \mathbb{E}(t,k)
=\sum_{k\ge 0} 4\pi q(t,k)k^2. \label{2.2a}
\end{align}
Here $\mathbb{E}(t,k), k=|\kk|$, is the energy wavenumber spectrum, $q(t,k)$ can be interpreted as
the density of contributions in wavenumber space
to the total energy, which is sometimes called the spectral density,
and $\kk=(k_1, k_2, k_3)=\frac{2\pi}{P}(n_1,n_2,n_3)\in \R^3$, with $n_j=0, \pm 1, \pm2, \cdots$, and $j=1,2,3$,
is the discrete wavevector in the Fourier transform:
\begin{eqnarray}\label{2.3a}
(\reallywidehat{\sqrt{\rho}\u},  \reallywidehat{\sqrt{\rho}c(\rho)})(t,\kk)
=\frac{1}{|\TP|}\int_{\TP}(\sqrt{\rho}\u, \sqrt{\rho}c(\rho))(t,\x)e^{-i\kk\cdot\x}\dd\x
\end{eqnarray}
of
the weighted velocity and sonic speed $(\sqrt{\rho}\u, \sqrt{\rho} c(\rho))(t,\x)$ in the $\x$--variable.
Then
$$
(\sqrt{\rho}\u, \sqrt{\rho}c(\rho))(t,\x)
=\sum_{\kk}(\reallywidehat{\sqrt{\rho}\u},
  \reallywidehat{\sqrt{\rho}c(\rho)})(t,\kk)e^{i\kk\cdot\x}
$$

\medskip
We adopt Kolmogorov's assumptions (1941) in his description of isotropic incompressible turbulence
in Kolmogorov \cite{Kolmogorov1,Kolmogorov2} (also see  McComb \cite{McComb}) to
introduce the following
compressible Kolmogorov-type hypothesis (CKH) in mathematical terms:

\bigskip
{\bf Assumption (CKH):} {\it For any $T>0$, there exist $M_T>0$ and $k_*$
{\rm (}sufficiently large{\rm )} independent of the viscosity $\mu$,
which may depend on $T$, $\mathcal{E}_0$, $\gamma>1$,
and $\|\f\|_{L^1(0, T; L^{\frac{2\gamma}{\gamma-1}}(\TP))}$,
such that,
for $k=|\kk|\ge k_*$,
\begin{equation} \label{assumption}
\int_0^T \mathbb{E}(t, k)\dd t \le M_T k^{-\frac{5}{3}}.
\end{equation}
}

\smallskip
For general turbulence, the energy wavenumber spectrum $\mathbb{E}(t,k)$ in \eqref{2.2a}
may be replaced by
$\mathbb{E}(t, k, \phi, \theta)$ in the spherical
coordinates $(k, \phi, \theta)$, $0\le \phi\le \pi$,
$0\le \theta \le 2\pi$, in the $\kk$-space,
but it should be in the same asymptotics as in \eqref{2.2a}
for sufficiently high wavenumber $k=|\kk|$.

\bigskip
For our analysis, the following weaker version of {\it Assumption} (CKHw) is sufficient:

\medskip
{\bf Assumption (CKHw):}
{\it For any $T>0$, there exist $M_T=C_T>0$ and $k_*$
{\rm (}sufficiently large{\rm )}
independent of the viscosity coefficient $\mu$ such that,
for $k=|\kk|\ge k_*$,
\begin{equation} \label{assumption-w}
\sup_{k\ge k_*}\Big(|\kk|^{3+\beta}\int_0^T
\big|(\reallywidehat{\sqrt{\rho}\u},
  \reallywidehat{\sqrt{\rho}c(\rho)})(t,\kk)\big|^2 \dd t\Big)
\le M_T  \qquad\mbox{for some $\beta>0$}.
\end{equation}
}

\medskip
In {\it Assumption} (CKHw), the case $\beta=\frac{2}{3}$ in \eqref{assumption-w}
corresponds to {\it Assumption} (CKH) in \eqref{assumption}.
{\it Assumption} (CKHi) corresponds the requirement that $\beta>\frac{2}{3}$.

As indicated in Chen-Glimm \cite{CG12}, a mathematical proof of \textit{Assumption} (CKH)
may well depend on developing a mathematical
version of the renormalization group, which
has proved to be very powerful in theoretical physics calculations.

\medskip
\section{$L^2$--Equicontinuity of the Weighted Velocity and Sonic Speed in the Space Variables,
Independent of the Viscosity}

\medskip
In this section, we show that the
compressible Kolmogorov-type hypothesis,
\textit{Assumption} (CKHw), implies
a uniform bound of the weighted velocity $(\sqrt{\rho^\mu}\u^\mu)(t,\x)$
and sonic speed $(\sqrt{\rho^\mu}c(\rho^\mu))(t,\x)$
in $L^2(0,T; H^\alpha(\TP))$ for any
$\alpha\in (0, \frac{\beta}{2})$, especially the uniform equicontinuity
of the weighted velocity $(\sqrt{\rho^\mu}\u^\mu)(t,\x)$ and the weighted sonic speed
$(\sqrt{\rho^\mu}c(\rho^\mu))(t,\x)$ in the space variables in $L^2([0,T)\times \TP)$,
independent of $\mu>0$.
This also implies the uniform equicontinuity
of the density and momentum in appropriate norms as specified below.

\begin{proposition}\label{prop:3.1}
Under {\rm Assumption (CKHw)}, for any $T\in (0, \infty)$,
there exists $M_T>0$,
independent of $\mu>0$,
such that
\begin{equation}\label{3.1}
\big\|(\sqrt{\rho^\mu}\u^\mu,
   \sqrt{\rho^\mu}c(\rho^\mu))
\big\|_{L^2(0,T; H^{\alpha}(\TP))}\le M_T<\infty,
\end{equation}
where $\alpha \in (0, \frac{\beta}{2})$.
\end{proposition}

\begin{proof}
Set $\w^\mu:=(\sqrt{\rho^\mu}\u^\mu,
  \sqrt{\rho^\mu}c(\rho^\mu))$.
Using the definition of fractional derivatives via the Fourier transform,
the Parseval identity, and {\it Assumption} (CKHw), {\it i.e.}, \eqref{assumption},
we have
\begin{eqnarray*}
&&\int_0^T\int_{\TP} |D^\alpha_\x \w^\mu(t,\x)|^2 \dd\x \dd t\\
&&\le M_1\int_0^T \Big(\sum_{\kk} |\kk|^{2\alpha}|\widehat{\w^\mu}(t,\kk)|^2 \Big) \dd t\\
&&=M_1\int_0^T \Big(\sum_{0\le|\kk|\le k_*} |\kk|^{2\alpha}|\widehat{\w^\mu}(t,\kk)|^2 \Big) \dd t
   +M_1\int_0^T \Big(\sum_{|\kk|>k_*} |\kk|^{2\alpha}|\widehat{\w^\mu}(t,\kk)|^2 \Big) \dd t\\
&&\le M_1 k_*^{2\alpha}\int_0^T \Big(\sum_{0\le|\kk|\le k_*}|\widehat{\w^\mu}(t,\kk)|^2 \Big) \dd t
   +M_2\sum_{|\kk|\ge k_*} |\kk|^{2\alpha-3-\beta}\\
&&\le M_1 k_*^{2\alpha}\int_0^T \Big(\sum_{\kk}|\widehat{\w^\mu}(t,\kk)|^2 \Big) \dd t
   +M_3\sum_{k\ge k_*} k^{2\alpha-1-\beta}\\
&&\le M_1k_*^{2\alpha}\int_0^T\int_{\TP}|\w^\mu(t,\x)|^2\dd\x \dd t
   +M_4 k_*^{2\alpha-\beta}\\
&&\le M_5k_*^{2\alpha}T\Big(\int_{\TP}|\w_0(\x)|^2\dd\x+\int_0^T\int_{\TP}|\f(t,\x)|^2\dd\x dt\Big)
+M_4 k_*^{2\alpha-\beta}\\
&&\le M_T^2<\infty,
\end{eqnarray*}
since  $\alpha< \frac{\beta}{2}$, which $M_j, j=1, \cdots, 5$, and $M_T$ are the constants independent of $\mu$,
but may depends on $T, k_*$, and $\alpha$.
This completes the proof.
\end{proof}

Proposition 3.1 directly yields the uniform equicontinuity of
$(\sqrt{\rho^\mu}\u^\mu(t,\x),\sqrt{\rho^\mu}c(\rho^\mu))$
in $\x$ in $L^2([0, T)\times\TP)$ independent of $\mu>0$, which particularly implies
the uniform equicontinuity  of the density $\rho^\mu$ in both space and time in $L^\gamma$ and the momentum $\m^\mu$ in $L^2$,
as well as the corresponding uniform high integrability of $(\sqrt{\rho^\mu}\u^\mu(t,\x),\sqrt{\rho^\mu}c(\rho^\mu))$, independent of $\mu>0$.

\begin{proposition}\label{coro:3.1}
Under {\rm Assumption (CKHw)}, for any $T\in (0, \infty)$,
there exists $M_T>0$,
independent of $\mu>0$,
such that
\begin{enumerate}
\item[\rm (i)] There exists $q=q(\beta)>2$ so that
\begin{equation}\label{equi-continuity-x-a}
\big\|(\sqrt{\rho^\mu}\u^\mu, \sqrt{\rho^\mu}c(\rho^\mu))\big\|_{L^q([0,T]\times \TP)}\le M_T;
\end{equation}

\item[\rm (ii)] There exists $\alpha_1\in (0, \frac{\beta}{4})$ so that
\begin{eqnarray}
&&\int_0^T\int_\TP\big|\rho^\mu(t,\x+\Delta\x)-\rho^\mu(t,\x)\big|^{\gamma}\dd\x \dd t\nonumber\\[1mm]
&&\,\,+\int_0^T\int_\TP\big|\m^\mu(t,\x+\Delta\x)-\m^\mu(t,\x)\big|^2\dd\x \dd t
\le M_T|\Delta\x|^{\alpha_1}. \label{equi-continuity-x-b}
\end{eqnarray}
In particular,
\begin{equation}\label{3.4a}
\big\|\m^\mu\big\|_{L^2([0,T]\times \TP)}\le M_T.
\end{equation}
\end{enumerate}
\end{proposition}

\begin{proof}
We divide the proof into three steps.

\smallskip
1.
The result in (i) directly follows by combining Proposition 3.1 with the Sobolev embedding theorem, since
$(\sqrt{\rho^\mu}\u^\mu, \sqrt{\rho^\mu}c(\rho^\mu))$ are uniformly bounded in $L^2(0,T; H^\alpha(\TP))$ for
$\alpha\in (0,\frac{\beta}{2})$.

\smallskip
2. Proposition 3.1 implies that
$$
\int_0^T\int_\TP \big|(\rho^\mu)^{\frac{\gamma}{2}}(t,\x+\Delta \x)-(\rho^\mu)^{\frac{\gamma}{2}}(t,\x)\big|^2\dd\x \dd t
\le M_T |\Delta \x|^{2\alpha}.
$$

When $\gamma\ge 2$, using that $|x-1|^\gamma\le (x^{\frac{\gamma}{2}}-1)^2$ for any $x\ge 0$,
we have
\begin{eqnarray}
&&\int_0^T\int_\TP\big|\rho^\mu(t,\x+\Delta\x)-\rho^\mu(t,\x)\big|^{\gamma}\dd\x \dd t\nonumber\\
&&\le \int_0^T\int_\TP \big|(\rho^\mu)^{\frac{\gamma}{2}}(t,\x+\Delta \x)-(\rho^\mu)^{\frac{\gamma}{2}}(t,\x)\big|^2\dd\x \dd t
\le M_T |\Delta \x|^{2\alpha}. \label{3.4b}
\end{eqnarray}

When $\gamma\le 2$, since
$$
x-y=\big(x^{\frac{\gamma}{2}}\big)^{\frac{2}{\gamma}}-\big(y^{\frac{\gamma}{2}}\big)^{\frac{2}{\gamma}}
=\frac{2}{\gamma}\int_0^1\big(\theta x^{\frac{\gamma}{2}}+(1-\theta)y^{\frac{\gamma}{2}})^{\frac{2}{\gamma}-1}\dd\theta
  |x^{\frac{\gamma}{2}}-y^{\frac{\gamma}{2}}|
$$
for any $x>0$ and $y>0$, we have
$$
|x-y|^\gamma\le C \big(x^{\frac{2-\gamma}{2}}+ y^{\frac{2-\gamma}{2}}\big)
 |x^{\frac{\gamma}{2}}-y^{\frac{\gamma}{2}}|^\gamma.
$$
Then
\begin{eqnarray}
&&\int_0^T\int_\TP\big|\rho^\mu(t,\x+\Delta\x)-\rho^\mu(t,\x)\big|^{\gamma}\dd\x \dd t\nonumber\\
&&\le M_T\Big(\int_0^T\int_\TP\big(\big|\rho^\mu(t,\x+\Delta\x)\big|+\big|\rho^\mu(t,\x)\big|\big)\,\dd\x \dd t\Big)^{\frac{2-\gamma}{2}}\nonumber \\
&&\qquad\,\,\, \times\Big(\int_0^T\int_\TP\big|(\rho^\mu)^{\frac{\gamma}{2}}(t,\x+\Delta\x)-(\rho^\mu)^{\frac{\gamma}{2}}(t,\x)\big|^2 \dd\x \dd t\Big)^{\frac{\gamma}{2}}\nonumber\\
&&\le M_T |\Delta\x|^{\alpha\gamma}.\label{3.5b}
\end{eqnarray}

3.  Proposition 3.1 and \eqref{3.4b}--\eqref{3.5b}
imply
\begin{eqnarray}
&&\int_0^T\int_\TP\big|\m^\mu(t,\x+\Delta\x)-\m^\mu(t,\x)\big|^2\dd\x \dd t\nonumber\\
&&\le  \|\sqrt{\rho^\mu}\|_{L^2(\TP)}
  \Big(\int_0^T\int_\TP\big(\big|(\sqrt{\rho^\mu}\u^\mu)(t,\x+\Delta\x)-(\sqrt{\rho^\mu}\u^\mu)(t,\x)\big|^2\big)\dd\x \dd t\Big)^{\frac{1}{2}}\nonumber\\
&&\quad  + \|\sqrt{\rho^\mu}\u^\mu\|_{L^2([0,T]\times \TP)}
  \Big(\int_0^T\int_\TP\big(\big|\sqrt{\rho^\mu}(t,\x+\Delta\x)-\sqrt{\rho^\mu}(t,\x)\big|\big)^2\dd\x \dd t\Big)^{\frac{1}{2}}\nonumber\\
&&\le M_T|\Delta \x|^{\alpha}
  +M_T\Big(\int_0^T\int_\TP \big|\rho^\mu(t,\x+\Delta\x)-\rho^\mu(t,\x)\big|^\gamma\, \dd\x \dd t\Big)^{\frac{1}{2\gamma}}\nonumber\\
&&\le M_T\big(|\Delta \x|^{\alpha} +|\Delta \x|^{\alpha \min\{\frac{1}{\gamma}, \frac{1}{2}\}}\big)\nonumber\\
&&\le M_T |\Delta \x|^{\alpha \min\{\frac{1}{\gamma}, \frac{1}{2}\}}. \label{3.6b}
\end{eqnarray}

Combining \eqref{3.4b}--\eqref{3.6b} together and choosing
$\alpha_1= \alpha \min\{\frac{1}{\gamma}, \frac{1}{2}\} \in (0, \frac{\beta}{4})$ yield \eqref{equi-continuity-x-b}.

In particular,
\begin{align*}
\int_0^T\int_\TP\big|\m^\mu(t,\x)\big|^2\,\dd\x \dd t
\le  \big\|\rho^\mu\big\|_{L^1([0,T]\times \TP)} \big\|\sqrt{\rho^\mu}\u^\mu\big\|_{L^2([0,T]\times \TP)}
\le  M_T<\infty.
\end{align*}
This completes the proof.
\end{proof}

\medskip
\section{$L^2$--Equicontinuity of the Density and Momentum in the Time Variable,
Independent of the Viscosity}

In this section, we show that Proposition 3.1 implies the uniform equicontinuity
of the density and momentum in the time variable $t>0$ in $L^2$, independent of $\mu>0$.

\begin{proposition}\label{prop:4.1}
For any $T>0$, there exist $\alpha_2=\alpha_2(\beta)>0$ depending only on $\beta$, and $M_T>0$
independent of $\mu>0$ such that,
for all small $\triangle t>0$,
\begin{eqnarray}\label{4.1}
&&\int_0^{T-\triangle t}\int_{\TP} \big|\rho^\mu (t+\triangle t, \x)-\rho^\mu(t,\x)\big|^\gamma \dd\x \dd t\nonumber\\
&&+\int_0^{T-\triangle t}\int_{\TP} \big|\m^\mu (t+\triangle t, \x)-\m^\mu(t,\x)\big|^2\dd\x \dd t\nonumber\\[2mm]
&&\le M_T(\triangle t)^{\alpha_2} \to 0 \qquad \mbox{as}\,\,
\triangle t\to 0.
\end{eqnarray}
\end{proposition}

\begin{proof} For simplicity, we drop the superscript $\mu>0$ of $(\rho^\mu, \m^\mu)$ in
the proof. Fix $\triangle t>0$.
We divide the proof into two steps.

\medskip
1.  For $t\in [0, T-\triangle t]$,
set
$$
w(t,\cdot):=\rho(t+\triangle t, \cdot)-\rho(t, \cdot).
$$
Then, for any $\varphi(t,\x)\in C^\infty([0,T)\times\TP)$ that is periodic in $\x\in\R^3$ with
period $\TP$, we have
\begin{align}
\int_{\TP} w(t,\x) \varphi(t,\x)\, \dd\x
&=\int_{t}^{t+\triangle t}\int_{\TP} \partial_s \rho(s,\x) \varphi(t,\x)\, \dd\x \dd s\nonumber\\
&=\int_t^{t+\triangle t}\int_\TP
 \m(s, \x)\cdot \nabla\varphi(t,\x) \, \dd\x \dd s.
\label{4.2}
\end{align}

By approximation, equality \eqref{4.2} still holds
for
$$
\varphi\in L^\infty(0,T; H^1(\TP))\cap C([0,T]; L^2(\TP))\cap L^\infty([0,T)\times\TP).
$$
Choose
\begin{equation}\label{4.3-phi-1}
\varphi=\varphi^\delta(t,\x):=(j_\delta*w)(t,\x)=\int
j_\delta(\x-\y) {\rm sign} (w(t,\y))|w(t,\y)|^{\gamma-1}\, \dd\y\in\R^3  \qquad \mbox{for $\delta>0$},
\end{equation}
which is periodic in $\x\in\R^3$ with period $\TP$, where $j_\delta(\x)=\frac{1}{\delta^3}j(\frac{\x}{\delta})\ge 0$
is a standard mollifier
with
\begin{equation}\label{4.3b}
j\in C_0^\infty(\R^3), \qquad \int j(\x)\dd\x=1.
\end{equation}
Then, for any $\x\in \TP$, we have
\begin{align}
\|\varphi^\delta\|_{L^\infty(\TP)}
&
\le \frac{1}{\delta^3}
\int_{|\x-\y|\le \delta} j(\frac{\x-\y}{\delta})|w(t,\y)|^{\gamma-1}\, \dd\y
\nonumber\\
&
\le \frac{M_T}{\delta^{3-\frac{3}{\gamma}}}\|j\|_{L^\gamma}\|w(t,\cdot)\|_{L^\gamma(\TP)}^{\gamma-1}
\le \frac{M_T}{\delta^{3-\frac{3}{\gamma}}}. \label{4.3-phi-2}
\end{align}
Similarly, we have
\begin{align}
\|\nabla\varphi^\delta\|_{L^\infty(\TP)}
&\le \frac{1}{\delta^4}
 \int_{|\x-\y|\le \delta} \big|j'(\frac{\x-\y}{\delta})\big||w(t,\y)|^{\gamma-1}\, \dd\y\nonumber\\
&\le \frac{M}{\delta^{4-\frac{3}{\gamma}}}\|j'\|_{L^\gamma}\|w(t,\cdot)\|_{L^\gamma(\TP)}^{\gamma-1}
\le \frac{M}{\delta^{4-\frac{3}{\gamma}}}.\label{4.3-phi-3}
\end{align}
As before, we use $M_T>0$ as a universal constant independent
of $\mu>0$.

Integrating \eqref{4.2} in $t$ over $[0, T-\triangle t)$ with
$\varphi=\varphi^\delta(t,\x)$, we have
\begin{align}
\int_0^{T-\triangle t}\int_{\TP} |w(t,\x)|^\gamma \dd\x \dd t
&=\int_0^{T-\triangle t}\int_t^{t+\triangle t}\int_{\TP}
\m(s,\x)\cdot\nabla\varphi^\delta (t,\x) \, \dd\x \dd s dt\nonumber\\
&\quad + \int_0^{T-\triangle t}\int_{\TP}
  \big(|w(t,\x)|^\gamma-w(t,\x)\varphi^\delta(t,\x)\big)\, \dd\x \dd t \nonumber\\
&=:I_1^\delta +I_2^\delta.
\label{4.2a}
\end{align}
Then
\begin{equation}\label{4.4}
|I_1^\delta|\le \frac{M_T\triangle t}{\delta^{4-\frac{3}{\gamma}}}\|\m(t,\cdot)\|_{L^1(\TP)}
\le \frac{M_T\triangle t}{\delta^{4-\frac{3}{\gamma}}}\|\rho\|_{L^1(\TP)}^{\frac{1}{2}}\|\sqrt{\rho}\u\|_{L^2(\TP)}
\le \frac{M_T\triangle t}{\delta^{4-\frac{3}{\gamma}}}.
\end{equation}

Furthermore, for $I_2^\delta$,
we have
\begin{align}
|I_2^\delta|&=\Big|\int_0^{T-\triangle t}\int_{\TP} \Big(\int_{\R^3}
j_\delta(\x-\y)\big(|w(t,\x)|^\gamma-w(t,\x){\rm sign}(w(t,\y))|w(t,\y)|^{\gamma-1}\big)\dd\y\Big) \dd\x \dd t\Big|
\nonumber\\
&=\Big|\int_{|\y|\le 1} j(\y)  \Big(\int_0^{T-\triangle t}\int_{\TP}
  \big(|w(t,\x)|^\gamma-w(t,\x){\rm sign}(w(t,\x-\delta \y))|w(t,\x-\delta \y)|^{\gamma-1}\big) \dd\x \dd t\Big)\dd\y\Big|
\nonumber\\
&\le \int_{|\y|\le 1} j(\y)  \Big(\int_0^{T-\triangle t}\int_{\TP}
  \big| |w(t,\x)|^\gamma-|w(t,\x-\delta \y)|^{\gamma}\big| \dd\x \dd t\Big)\dd\y
\nonumber\\
&\quad +\Big|\int_{|\y|\le 1} j(\y)  \Big(\int_0^{T-\triangle t}\int_{\TP}
  \big(w(t,\x)-w(t,\x-\delta \y)\big){\rm sign}(w(y,\x-\delta\y))|w(t,\x-\delta \y)|^{\gamma-1} \dd\x \dd t\Big)\dd\y\Big|
\nonumber\\
&\le \gamma
\int_{|\y|\le 1} j(\y)  \Big(\int_0^{T-\triangle t}\int_{\TP}
  |w(t,\x)-w(t,\x-\delta \y)| (|w(t,\x)|^{\gamma-1}+|w(t,\x-\delta \y)|^{\gamma-1}) \dd\x \dd t\Big)\dd\y
\nonumber\\
&\quad +\int_{|\y|\le 1} j(\y)  \Big(\int_0^{T-\triangle t}\int_{\TP}
  |w(t,\x)-w(t,\x-\delta \y)| |w(t,\x-\delta \y)|^{\gamma-1} \dd\x \dd t\Big)\dd\y
\nonumber\\
&\le
M_T \|w\|^{\gamma-1}_{L^\gamma([0,T]\times \TP)}\int_{|\y|\le 1} j(\y)  \Big(\int_0^{T-\triangle t}\int_{\TP}
  |w(t,\x)-w(t,\x-\delta \y)|^\gamma  \dd\x \dd t\Big)^{\frac{1}{\gamma}}
    \dd\y \nonumber\\
&\le M_T\delta^{\frac{\alpha_1}{\gamma}}\int_{|\y|\le 1} j(\y)|\y|^{\frac{\alpha_1}{\gamma}}\le M_T\delta^{\frac{\alpha_1}{\gamma}},
\label{4.6}
\end{align}
where we have used the fact that $\|w\|_{L^\gamma([0,T]\times \TP)}\le M_T$ from
Proposition \ref{coro:3.1}.

Combining \eqref{4.2}--\eqref{4.6}, we have
\begin{eqnarray*}
\int_0^{T-\triangle t}\int_{\TP} |\w(t,\x)|^\gamma \dd\x \dd t
\le M\inf_{\delta>0} \big\{\frac{\triangle t}{\delta^{4-\frac{3}{\gamma}}}+\delta^{\frac{\alpha_1}{\gamma}}\big\}.
\end{eqnarray*}
Choose
$$
\delta=|\triangle t|^{\frac{\gamma}{4\gamma+ \alpha_1-3}}.
$$
We conclude
\begin{equation}
\int_0^{T-\triangle t}\int_{\TP} |\w(t,\x)|^\gamma \dd\x \dd t
\le M (\triangle t)^{\frac{\alpha_1}{4\gamma-3+ \alpha_1}}.
\end{equation}

\medskip
2.  Now we estimate $\m(t,\x)$. For $t\in [0, T-\triangle t]$,
set
$$
\w(t,\cdot)=\m(t+\triangle t, \cdot)-\m(t, \cdot).
$$
Then, for any $\varphib(t,\x)\in C^\infty([0,T)\times\TP)$ that is periodic in $\x\in\R^3$ with
period $\TP$, we have
\begin{align}
\int_{\TP}\w(t,\x)\cdot \varphib(t,\x)\, \dd\x
&=\int_{t}^{t+\triangle t}\int_{\TP} \partial_s\m(s,\x)\cdot \varphib(t,\x)\, \dd\x \dd s\nonumber\\
&=\int_t^{t+\triangle t}\int_\TP
 (\rho \u\otimes\u)(s, \x): \nabla\varphib(t,\x) \, \dd\x \dd s \nonumber\\
& \quad - \int_t^{t+\triangle t}\int_{\TP} \Sigma(s,\x) :\nabla\varphib(t,\x)\,  \dd\x \dd s \nonumber\\
&\quad +\int_t^{t+\triangle t}\int_{\TP} p(s,\x) \nabla\cdot\varphib(t,\x)\, \dd\x
\dd s\nonumber\\
&\quad +\int_t^{t+\Delta t}\int_\TP (\rho\f)(s,\x)\cdot \varphib(t,\x) \dd\x \dd s,
\label{4.2c}
\end{align}
where $\nabla\varphib=(\partial_{x_j}\varphi_i)$ is the $3\times 3$ matrix.

By approximation, equality \eqref{4.2c} still holds
for
$$
\varphib\in L^\infty(0,T; H^1(\TP))\cap C([0,T]; L^2(\TP))\cap L^\infty([0,T)\times\TP).
$$
Choose
$$
\varphib=\varphib^\delta(t,\x):=(j_\delta*\w)(t,\x)=\int
j_\delta(\y)\w(t,\x-\y)\, \dd\y\in\R^3 \qquad\mbox{for $\delta>0$},
$$
which is periodic in $\x\in\R^3$ with period $\TP$, where $j_\delta(\x)=\frac{1}{\delta^3}j(\frac{\x}{\delta})\ge 0$
is the standard mollifier with \eqref{4.3b}.

Then, for any $\x\in \TP$, we have
\begin{align}
\|\varphib_\delta\|_{L^\infty(\TP)}
&\le  \frac{1}{\delta^3}
\int_{|\x-\y|\le \delta} j(\frac{\x-\y}{\delta})|\w(t,\y)|\, \dd\y
\le \frac{M_T}{\delta^{3}}\|j\|_{L^\infty}\|\sqrt{\rho(t,\cdot)}\, (\sqrt{\rho}\u)(t,\cdot)\|_{L^1(\TP)}\nonumber\\[1mm]
&\le \frac{M_T}{\delta^{3}}\|j\|_{L^\infty}\|\rho(t,\cdot)\|_{L^1}\|(\sqrt{\rho}\u)(t,\cdot)\|_{L^2(\TP)}
\le \frac{M_T}{\delta^{3}}.\label{4.3-phi-2a}
\end{align}
Similarly, we have
\begin{align}
\|\nabla\varphi_\delta\|_{L^\infty(\TP)}
\le\, \frac{1}{\delta^4}
\int_{|\x-\y|\le \delta} \big|j'(\frac{\x-\y}{\delta})\big| |\w(t,\y)|\, \dd\y
\le\, \frac{M_T}{\delta^{4}}\|j'\|_{L^\infty}\|\w(t,\cdot)\|_{L^1(\TP)}
\le\, \frac{M_T}{\delta^{4}}.\label{4.3-phi-3a}
\end{align}

Integrating \eqref{4.2c} in $t$ over $[0, T-\triangle t)$ with
$\varphib=\varphib^\delta(t,\x)$ and using \eqref{4.3-phi-3a}, we have
\begin{align}
&\int_0^{T-\triangle t}\int_{\TP} |\w(t,\x)|^2 \dd\x \dd t
\nonumber\\
&=\int_0^{T-\triangle t}\int_{\TP} \w(t,\x)\cdot \big(\w(t,\x)
-\varphib^\delta(t,\x)\big)\, \dd\x \dd t \nonumber\\
&\quad +\int_0^{T-\triangle t}\int_t^{t+\triangle t}\int_{\TP}
(\rho\u\otimes\u)(s,\x): \nabla\varphib^\delta (t,\x) \, \dd\x \dd s dt\nonumber\\
&\quad - \int_0^{T-\triangle t}\int_{\TP}\int_t^{t+\Delta t}
\Sigma(s,\x): \nabla\varphib^\delta(t,\x)\, \dd s \dd\x dt\nonumber\\
&\quad + \int_0^{T-\triangle t}\int_{\TP}\int_t^{t+\Delta t}
p(s,\x) \nabla\cdot\varphib^\delta(t,\x)\, \dd s \dd\x dt\nonumber\\
&\quad +\int_0^{T-\Delta t}\int_t^{t+\Delta t}\int_\TP
(\rho\f)(s,\x)\cdot \varphib^\delta(t,\x)\, \dd\x \dd s \dd t\nonumber\\
&=:J_1^\delta +J_2^\delta+J_3^\delta +J_4^\delta+J_5^\delta.
\label{4.2d}
\end{align}

For $J_1^\delta$, we have
\begin{align}
|J_1^\delta|&=\Big|\int_0^{T-\triangle t}\int_{\TP} \Big(\int_{\R^3}
j_\delta(\x-\y)\big(|\w(t,\x)|^2-\w(t,\x)\w(t,\y)\big)\dd\y\Big) \dd\x \dd t\Big|
\nonumber\\
&=\int_{|\y|\le 1} j(\y)  \Big(\int_0^{T-\triangle t}\int_{\TP}
  \big(|\w(t,\x)|^2-\w(t,\x)\w(t,\x-\delta \y)\big) \dd\x \dd t\Big)\dd\y
\nonumber\\
&\le \int_{|\y|\le 1} j(\y)  \Big(\int_0^{T-\triangle t}\int_{\TP}
  |\w(t,\x)-\w(t,\x-\delta \y)| \big(|\w(t,\x)|+|\w(t,\x-\delta \y)|\big) \dd\x \dd t\Big)\dd\y
\nonumber\\
&\quad +\int_{|\y|\le 1} j(\y)  \Big(\int_0^{T-\triangle t}\int_{\TP}
  |\w(t,\x)-\w(t,\x-\delta \y)| |\w(t,\x-\delta \y)|\dd\x \dd t\Big)\dd\y
\nonumber\\
&\le
M_T\|\w\|_{L^2([0,T]\times \TP)}\int_{|\y|\le 1} j(\y)  \Big(\int_0^{T-\triangle t}\int_{\TP}
  |\w(t,\x)-\w(t,\x-\delta \y)|^2  \dd\x \dd t\Big)^{\frac{1}{2}}
    \dd\y
\nonumber\\
&\le M_T\delta^{\frac{\alpha_1}{2}}\int_{|\y|\le 1}j(\y)|\y|^{\frac{\alpha}{2}}d\y
\le M_T\delta^{\frac{\alpha_1}{2}},
\label{4.4c}
\end{align}
where we have used the fact that $\|\w\|_{L^2([0,T]\times \TP)}\le M_T$ from
Proposition \ref{coro:3.1}.

For $J_2^\delta$ and $J_3^\delta$, we find
\begin{equation}\label{4.5c}
|J_2^\delta|
\le \frac{M_T\Delta t}{\delta^4}\|\sqrt{\rho} \u\|^2_{L^1(\TP)}
\le  \frac{M_T\Delta t}{\delta^4}\mathcal{E}_0\le \frac{M_T\Delta t}{\delta^4},
\end{equation}
and
\begin{align}\label{4.6c}
|J_3^\delta|
&\le \frac{M_T}{\delta^4}
\int_0^{T-\triangle t}\int_{\TP}\int_t^{t+\triangle t}
|\Sigma| \,\dd\x \dd s \dd t\nonumber\\
&\le \frac{M_T\Delta t}{\delta^4}
\int_0^{T-\triangle t}\int_{\TP}
|\Sigma| \,\dd\x \dd s\nonumber\\
&\le \frac{M_T\Delta t}{\delta^4}
\Big(\int_0^{T-\triangle t}\int_{\TP}
|\Sigma|^2 \dd\x \dd s\Big)^2 \nonumber\\
&
\le \frac{M_T\Delta t}{\delta^4},
\end{align}
since $\max\{\mu, |\lambda|\}\le \mu_0$ for some fixed $\mu_0$ throughout the paper.

For $J_4^\delta$ and $J_5^\delta$, we see
\begin{eqnarray}\label{4.7c}
|J_4^\delta|
\le \frac{M_T\Delta t}{\delta^4}\|\rho\|_{L^\gamma([0,T]\times \TP)}^{\gamma}
\le \frac{M_T\Delta t}{\delta^4},
\end{eqnarray}
and
\begin{eqnarray}\label{4.8c}
|J_5^\delta|
\le M_T\Delta t \|\rho \f\|_{L^1}\|\varphi^\delta\|_{L^\infty}
\le \frac{M_T\Delta t}{\delta^{3}}\|\rho\|_{L^\gamma(\TP)}\|\f\|_{L^1(0,T; L^{\frac{2\gamma}{\gamma-1}}(\TP))}
\le \frac{M_T\Delta t}{\delta^{3}}.
\end{eqnarray}

Therefore, combining \eqref{4.2d}--\eqref{4.8c} together, we have
$$
\int_0^{T-\triangle t}\int_{\TP}
|\w(t,\x)|^2 \dd\x \dd t
\le M_T\big(\delta^{\frac{\alpha_1}{2}}+\frac{\Delta t}{\delta^4}+ \frac{\Delta t}{\delta^{3}}\big)
\le M_T\big(\delta^{\frac{\alpha_1}{2}}+\frac{\Delta t}{\delta^4}\big)
\le M_T(\Delta t)^{\frac{\alpha_1}{8+\alpha_1}}.
$$
Choosing
$$
\alpha_2=\min\{\frac{\alpha_1}{4\gamma-3+ \alpha_1}, \frac{\alpha_1}{8+\alpha_1}\},
$$
we complete the proof.
\end{proof}

As a direct corollary, we have

\begin{proposition}\label{prop:4.2}
Under {\rm Assumption (CKHw)}, for any $T\in (0, \infty)$,
there exists $M_T>0$,
independent of $\mu>0$,
such that
\begin{equation}\label{4.8}
\int_0^T\int_{\TP} \Big(\big|(D^{\frac{\alpha_2}{\gamma}}_t, D^{\frac{\alpha_1}{\gamma}}_\x)\rho^\mu(t,\x)\big|^\gamma
 + \big|(D^{\frac{\alpha_2}{2}}_t, D^{\frac{\alpha_1}{2}}_\x)  \m^\mu(t,\x)\big|^2\Big) \dd\x \dd t \le M_T<\infty
\end{equation}
for some $\alpha_1>0$ and $\alpha_2=\min\{\frac{\alpha_1}{4\gamma-3+ \alpha_1}, \frac{\alpha_1}{8+\alpha_1}\}>0$.
\end{proposition}

\medskip

Combining Proposition 3.1 with Proposition 4.2, we have
\begin{equation}\label{4.9}
\big\|\rho^\mu\big\|_{W^{\hat{\alpha}_1,\gamma}([0, T)\times\TP)}+\big\|\m^\mu\big\|_{H^{\hat{\alpha}_2}([0, T)\times\TP)}\le M_T<\infty,
\end{equation}
where $\hat{\alpha}_1=\frac{1}{\gamma}\min\{\alpha_1,\alpha_2\}$
and $\hat{\alpha}_2=\frac{1}{2}\min\{\alpha_1,\alpha_2\}$.
Then, by the Sobolev imbedding theorem, we have

\begin{proposition}\label{prop:4.3}
Under {\rm Assumption (CKHw)},  for any $T\in (0, \infty)$,
there exists $M_T>0$,
independent of $\mu>0$,
such that
\begin{equation}\label{3.2}
\big\|\rho^\mu\big\|_{{L^{q_1}([0, T)\times\TP)}}+  \big\|\m^\mu\big\|_{L^{q_2}([0, T)\times\TP)}\le M_T<\infty
\end{equation}
for some $q_1>\gamma$ and $q_2>2$.
\end{proposition}

\medskip

\section{Inviscid Limit}
In this section, we show the inviscid limit from the compressible Navier-Stokes to the Euler equations
under \textit{Assumption} (CKHw).

\begin{theorem}\label{thm:5.1}
The
compressible Kolmogorov-type hypothesis, {\rm Assumption (CKHw)}, implies the strong compactness in
$(L^{q_1}\times L^{q_2})([0, T)\times\TP)$
of the solutions $(\rho^\mu, \m^\mu)(t,\x)$ of the compressible Navier-Stoke equations \eqref{eq:ns} in $\R^3_T$
in the sense of Definition {\rm \ref{dfn:ns}} when
the viscosity $\mu$ tends to zero, for some $q_1>\gamma$ and $q_2>2$.
That is,
there exist a subsequence {\rm (}still denoted{\rm )} $(\rho^\mu, \m^\mu)(t,\x)$
and a function $(\rho, \m)\in  (L^{q_1}\times L^{q_2})([0,T)\times\TP)$
such that
$$
(\rho^\mu, \m^\mu) (t, \x)\to (\rho, \m)(t,\x) \qquad \mbox{{\it a.e.} as } \, \mu\to 0,
$$
and $(\rho, \m)(t,\x)$ is a weak solution of the
compressible Euler equations \eqref{eq:euler} with Cauchy data $(\rho_0,\m_0)(\x)$ in the sense
of Definition {\rm \ref{dfn:euler}}.
\end{theorem}

\medskip

\begin{proof}  We divide the proof into six steps.

\medskip
1.  Propositions 3.2 and 4.1 imply the $(L^\gamma \times L^2)$--equicontinuity of $(\rho^\mu, \m^\mu)(t,\x)$
in $(t,\x)\in [0, T)\times\TP$, independent of $\mu>0$.
This yields that there exist both a subsequence (still denoted) $(\rho^\mu,\m^\mu)(t,\x)$
and a function $(\rho, \m)(t,\x)\in L^\gamma\times L^2$ such that
$$
(\rho^\mu, \m^\mu)(t,\x)\to (\rho, \m)(t,\x)      \qquad \mbox{in}\,\, L^\gamma\times L^2 \,\,\mbox{as $\mu\to 0$},
$$
which implies that
$$
(\rho^\mu,\m^\mu)(t,\x)\to (\rho, \m)(t,\x)      \qquad \mbox{\it a.e.}\,\,\mbox{as $\mu\to 0$}.
$$

\medskip
2. From Proposition \ref{prop:4.3}, we have
$$
(\rho, \m)\in  (L^{q_1}\times L^{q_2})([0, T)\times\TP),
$$
and
$$
(\rho^\mu, \m^\mu)(t,\x)\to (\rho,\m)(t,\x)
 \qquad \mbox{in}\,\, (L^{p_1}\times L^{p_2})([0, T)\times\TP) \,\,\mbox{as $\mu\to 0$}.
$$
for any $p_1\in [1, q_1)$ and $p_2\in [1, q_2)$.

\medskip
3. Proposition 3.2 implies that
$$
\M^\mu:=\rho^\mu \u^\mu\otimes \u^\mu  \rightharpoonup \M\in L^q
\qquad \mbox{weakly in $L^q$ for some $q>1$}.
$$

Since
$$
\m^\mu\otimes \m^\mu=\M^\mu \rho^\mu,
$$
then the limits satisfy
$$
\m\otimes \m=\M\rho \qquad a.e.
$$
which implies that
$\m\otimes \m$ is absolutely continuous with respect to $\rho$ with an $L^q$ density, uniformly in $\rho(t,\x)\ge 0$.
That is, $\frac{\m\otimes \m}{\rho}$ is well-defined even on the vacuum $\{\rho(t,\x)=0\}$, and
$\frac{\m\otimes \m}{\rho}=\M$ {\it a.e.} on the non-vacuum state $\{\rho(t,\x)>0\}$.
Therefore, we define
$$
\frac{\m\otimes \m}{\rho}=\M \in L^q \qquad \mbox{for $q>1$}.
$$

This implies that
\begin{equation}\label{5.1}
\frac{\m^\mu\otimes \m^\mu}{\rho^\mu}(t,\x)\to \frac{\m\otimes\m}{\rho}(t,\x)\in L^q  \qquad \mbox{\it a.e.}\,\,\mbox{as $\mu\to 0$}
\end{equation}
for $q>1$.

\medskip
4. Similarly, we notice that
$$
V^\mu:=\rho^\mu  |\u^\mu|^2  \to V\in L^q
\qquad \mbox{weakly in $L^q, q>1$}.
$$

Since
$$
|\m^\mu|^2=V^\mu\rho^\mu,
$$
then the limits satisfy
$$
|\m|^2=V\rho \qquad a.e.
$$
which implies that
$|\m|^2$ is absolutely continuous with respect to $\rho$ with an $L^q$ density, uniformly in $\rho(t,\x)\ge 0$.
That is, $\frac{|\m|^2}{\rho}$ is well-defined even on the vacuum $\{\rho(t,\x)=0\}$, and
$\frac{|\m|^2}{\rho}=V$ {\it a.e.} on the non-vacuum state $\{\rho(t,\x)>0\}$.
Therefore, we define
$$
\frac{|\m|^2}{\rho}=V\in L^q \qquad \mbox{for $q>1$}.
$$
This implies that
\begin{equation}\label{5.1b}
\frac{|\m^\mu|^2}{\rho^\mu}(t,\x)\to \frac{|\m|^2}{\rho}(t,\x)\in L^q  \qquad \mbox{\it a.e.}\,\,\mbox{as $\mu\to 0$}
\end{equation}
for $q>1$.

\medskip
5. For any $\varphi\in C_0^\infty(\R_T^3; \R)$,
we multiply  both sides \eqref{eq:ns} by $\varphi$ and integrate over $\R_T^3$ to
obtain
\begin{eqnarray*}
\int_0^T\int_{\R^3}\big(\rho^\mu \varphi_t + \m^\mu\cdot \nabla \varphi
      \big)\, \dd\x \dd t
+\int_{\R^3} \rho_0(\x) \varphi(0,\x)\, \dd\x=0.
\end{eqnarray*}
Letting $\mu\to 0$, we conclude \eqref{euler-1}.

6. For any $\varphib\in C_0^\infty(\R_T^3;\R^3)$,
we multiply  both sides \eqref{eq:ns} by $\varphib$ and integrate over $\R_T^3$ to
obtain
\begin{eqnarray*}
&&\int_0^T\int_{\R^3}\big(\m^\mu\cdot \varphib_t + \frac{\m^\mu\otimes\m^\mu}{\rho^\mu}:\nabla \varphib
+ p(\rho^\mu) \div \varphi
+\rho^\mu \f\cdot \varphib\big)\, \dd\x \dd t
+\int_{\R^3} \m_0(\x)\cdot \varphib(0,\x)\, \dd\x\\
&&=-\int_0^T\int_{\R^3} \Sigma^\mu: \nabla \varphib\, \dd\x \dd t\\
&&=-2\mu \int_0^T\int_{\R^3} D(\nabla \u^\mu): \nabla \varphib\, \dd\x \dd t
   -\lambda\int_0^T\int_{\R^3} (\nabla\cdot\u^\mu)^2 {\rm Tr}(\nabla\varphib)\, \dd\x \dd t.
\end{eqnarray*}
Notice that
\begin{eqnarray*}
&&\big|\mu \int_0^T\int_{\R^3} D(\nabla\u^\mu): \nabla \varphib\, \dd\x \dd t\big|
\le M_T\sqrt{\mu}\,\big\|\sqrt{\mu}\nabla\u^\mu\big\|_{L^2}\|\nabla\varphib\|^2_{L^2}\le \sqrt{\mu}\,M_T,\\
&&\big|\lambda\int_0^T\int_{\R^3} (\nabla\cdot\u^\mu)^2 {\rm Tr}(\Delta\varphib) \, \dd\x \dd t\big|
\le M_T\sqrt{|\lambda|}\big\|\sqrt{|\lambda|}\nabla\cdot\u^\mu\big\|_{L^2}\|\nabla\varphib\|^2_{L^2}\le \sqrt{|\lambda|}\,M_T.
\end{eqnarray*}
Letting $\mu\to 0$, we conclude \eqref{euler-2}.

\medskip
Integrating \eqref{energy-ineq} over $[0, T)\times\TP$ yields
\begin{align*}
\int_{\TP} E^\mu(t,\x) \dd\x
&=\int_{\TP} \big(\frac{1}{2}\frac{|\m^\mu|^2}{\rho}+\frac{\rho^\mu c(\rho^\mu)}{\gamma-1}\big)(t,\x) \dd\x\nonumber\\
&\le \int_{\TP} E_0(\x) \dd\x
+\int_0^t\int_\TP \m^\mu(s,\x)\cdot\f(s,\x) \dd\x \dd s.
\end{align*}
Letting $\mu\to 0$, we conclude \eqref{euler-3}.

Combining \eqref{euler-1}--\eqref{euler-3} together yields that $(\rho, \m)(t,\x)$ is a weak
solution of the Euler equations \eqref{eq:euler} with Cauchy
data \eqref{1.2} in the sense of Definition \ref{dfn:euler}.
This completes the proof.
\end{proof}

\section{Comments on the Main Results from a Physics Perspective}

While the extension of these results to general $\beta > 0$ is interesting,
the deepest conclusions result in the extensions beyond the K41 exponent
$-\frac{5}{3}$. For incompressible turbulence,
such $\beta$ values lead to a fractal dimension $d_f < 3$ of turbulence.
In other words, turbulence occurs on a fractal set of dimension strictly
less than $3$. In fact, there are two fractal sets for turbulence,
associated with the strain rate and the enthalpy
({\it cf.} \cite{GliMah18}),
with values
$d_f = 2.61 \pm 0.03$ for the strain rate fractal and
$d_f = 2.22 \pm 0.03$ for the enstrophy fractal.  These estimates
result from the analysis of the finely resolved turbulence data base
({\it cf.} \cite{LiPerWan08,PerBurLi07}).
The value for the
strain rate fractal is consistent with a variety of simulation and
experimental values reported by
others ({\it e.g.} \cite{Frisch}),
while the enstrophy fractal has received less prior attention ({\it cf.} \cite{Bra91}).
Multifractal theories ({\it cf.} \cite{Frisch})
for incompressible turbulence identify a continuous infinity of additional
fractal sets derived from the basic ones
of strictly decreasing fractal dimensions
and associated with the higher order statistics of the two point correlation
function.
Multifractal ideas motivate a number of simulation studies of compressible
turbulence, as discussed in the introduction.

It is a strict mathematical
result for the incompressible Euler equations with (CKHi) ({\it i.e.},
$\beta > \frac{2}{3}$, corresponding to the Kolmogorov exponent $\frac{5}{3}$)
that the total system energy is constant as a function
of time ({\it cf.} \cite{CCFS}).

An extension of the Kolmogorov hypothesis (K62) in \cite {Kolmogorov3}
and its corrections to compressible turbulence is a
plausible conjecture, and adds weight to our belief in the existence and
uniqueness of the compressible Euler equations in a regime of fully
developed turbulence. We do not assert that these hypotheses apply to
all solutions of the compressible Euler equations.
The Taylor-Green vortex, which does not possess
a K41 cascade of turbulent scales, but
rather consists of
a single large vortex, is postulated
\cite{MajBer01}
to be a fundamental singularity for the Euler equations in time $t < \infty$.
Reasoning against
this singularity is found in \cite{Fri03}.
Recent high resolution simulation on the solution of the
Taylor-Green vortex problem
can be found in \cite{DeBonis} (also see \cite{Bra91}).

The simulation does not show a singularity, but rather an eventual transition
from the isolated vortex to a generalized turbulent flow, as predicted
in \cite{Fri03}. The strictly axisymmetric vortex is unstable and, as the
instability is approached, it is increasingly unstable to non-axisymmetric
perturbations. As an isolated unstable point singularity in the
space of initial conditions, it is ``removable'', and the solution,
starting exactly at this isolated point of initial conditions,
can be extended (nonuniquely) beyond  the singularity.

A related and highly analyzed isolated point of nonuniqueness is given by the Rayleigh-Taylor
instability. A planar surface separates an incompressible
heavy fluid (above) from a light fluid (below). This initial condition is an
isolated unstable point in the space of initial conditions, in that small
perturbations will diverge from it, with the solution depending on the
perturbation. The Rayleigh-Taylor initial conditions are statically
unstable, while the Taylor-Green vortex is dynamically unstable, with its
degree of instability diverging as the critical time approaches.

While the Rayleigh-Taylor static instability point does not require any
reformulation of the uniqueness concepts of solutions, the Taylor-Green
isolated point of dynamic instability should be considered in any formulation
of uniqueness for the solutions of the Euler equations.

Nonuniqueness of solutions of the Euler equations has been established
mathematically for incompressible flow in \cite{DeLS1,DeLS2,Sch93},
and for compressible flow in \cite{CDK,CK,KM} and the references cited therein.
These solutions are very sparse in their
$k$ space representation, and are far removed from any type of energy
cascade and scaling law which could be associated with the K41 turbulence or its
corrections.
From the perspective of fundamental physics, the physical solution should be the one that maximizes
the rate of dissipation of entropy among all possible solutions of the Euler equations.
We emphasize that entropy has the familiar thermodynamic aspect concerning molecular motions ({\it i.e.}, heat released), but
it also includes turbulent fluid fluctuations.
Such a condition is inconvenient to apply and leaves open the statement of a convenient entropy principle.
As indicated earlier, for the incompressible Euler equations, (CKHi) implies constant energy in time for the solution ({\it cf.} \cite{CCFS}).
For the barotropic compressible Euler equations, the energy cannot be constant as it is dissipated (decreases) within a shock wave.
Thus, we propose as a necessary admissibility condition for physically meaningful solutions
of the Euler equations that they should be a strong limit of the Navier-Stokes solutions
with a uniform bound (CKHi).
To our knowledge, the wild solutions do not satisfy either of these conditions.
We note that intermittency bounds as in (CKHi) are insufficient to force maximization of the entropy decay rate,
due to higher order intermittency structures contributing to the entropy. We do not know of a convenient reformulation of
this condition.

At a deeper level, we comment on difficulties with the principle of maximum rate of entropy production as a strengthened
statement of the second law of thermodynamics.
Assessment of this principle for the purpose of validation or verification depends on solutions of a variational principle.
The variational principles for conservative equations ($\delta \mathcal{L}=0$, where $\mathcal{L}$ is the system Lagrangian)
and for dissipative systems (an entropy rate condition) are known,
but a gap in the formulation of mechanics is the absence of a variational formulation for systems that
combine conservative and dissipative components.

The numerical simulation model called Implicit Large Eddy Simulation (ILES) for turbulent flows has a sizable following,
although controversial within the larger turbulent simulation community.
By design, it is located exactly at the marginal decay rate $-\frac{5}{3}$, and further, by design, the associated prefactor
is minimized, so that it is smaller than $\epsilon^{\frac{2}{3}}$.
In an analysis to be published separately, we will present reasons for believing that the ILES solutions fail to maximize
the entropy dissipation rate among the competing (nonunique) solutions of the Euler equations,
or in other words,  that ILES are not physically meaningful.

\bigskip
\smallskip
{\bf Acknowledgments.}
$\,$  The research of
Gui-Qiang G. Chen was supported in part by
the UK
Engineering and Physical Sciences Research Council Award
EP/E035027/1 and
EP/L015811/1, and the Royal Society--Wolfson Research Merit Award (UK).
James Glimm's research was supported in part by
Stanford University (PSAAP/DOE) funding.

\end{document}